\documentclass{article}

\usepackage[utf8]{inputenc}
\usepackage[english]{babel}
\usepackage{amsfonts}
\usepackage[11pt]{moresize}
\usepackage{pgfplots}
\usepackage{xfrac}
\newcommand{\Addresses}{{
		\bigskip
		\footnotesize
		
		Goethe Universit\"at Frankfurt am Main, Institut f\"ur Mathematik, Robert-Mayer Strasse 6-8
		\hfill \newline\texttt{}
		\indent 60325 Frankfurt am Main, Germany} 
	\par\nopagebreak
	\textit{E-mail address}: \texttt{andreibud95@protonmail.com}
}

\usepackage{amssymb}
\usepackage{mathtools}
\usepackage{tikz}
\usepackage{tikz-cd}
\usepackage{mathtools}
\usepackage{amsmath}
\usepackage{amsthm}
\usepackage{hyperref}
\usepackage{float}
\usetikzlibrary{calc, positioning,decorations.markings,arrows}

\usepackage{graphics}
\usepackage{color}
\usepackage{tabularx,colortbl}
\usetikzlibrary{positioning}
\usetikzlibrary{decorations.markings,arrows}
\usetikzlibrary{calc}
\usepackage{amsthm}

\theoremstyle{plain}
\newtheorem{trm}{Theorem}[section]
\newtheorem{lm}[trm]{Lemma}
\newtheorem{prop}[trm]{Proposition}
\newtheorem{cor}[trm]{Corollary}

\theoremstyle{definition}

\def\OO{\mathcal{O}}

\def\cM{\mathcal{M}}
\def\cR{\mathcal{R}}
\def\rr{\overline{\mathcal{R}}}

\def\Pic0{{\rm Pic}^0(X)}

\def\mm{\overline{\mathcal{M}}}

\begin{document}
\title{Brill-Noether loci and strata of differentials}
\author{Andrei Bud}
\date{}
\maketitle
\begin{abstract}
	We prove that the projectivized strata of differentials are not contained in pointed Brill-Noether divisors, with only a few exceptions. For a generic element in a stratum of differentials, we show that many of the associated pointed Brill-Noether loci are of expected dimension. We use our results to study the Auel-Haburcak Conjecture: We obtain new non-containments between maximal Brill-Noether loci in $\mathcal{M}_g$. Our results regarding quadratic differentials imply that the quadratic strata in genus $6$ are uniruled.     
\end{abstract}
 
\section{Introduction}
 
 Starting with the work of Kontsevich and Zorich, see \cite{KontsevichZorich}, and of Polischuk, see \cite{polischuckr-spin}, the study of strata of $k$-differentials became an important research topic in Algebraic Geometry. Following these developments, many results about the compactification of strata (c.f. \cite{FP18}, \cite{DaweiAbComp}, \cite{Daweik-diffcomp} and \cite{Daweimulti-scale}), about the number of connected components of strata (c.f. \cite{KontsevichZorich}, \cite{Boissy}, \cite{Lanneau}, \cite{ChenM} and \cite{chengendron}), and about its class in the tautological ring of $\overline{\mathcal{M}}_{g,n}$ (c.f. \cite{Pandha-tautological-Picard}) have been obtained. An important aspect in the study of strata of differentials is understanding their birational geometry. 

Our goal in this paper is to understand the geometry of Brill-Noether loci for the generic curve appearing in a stratum of differentials. For $\mu = (m_1,\ldots, m_n)$ a partition of $k\cdot (2g-2)$, we define the projectivized strata of $k$-differentials 
\[\mathcal{H}_g^k(\mu) \coloneqq \left\{ [C,p_1,\ldots, p_n] \in \mathcal{M}_{g,n} \ | \ \mathcal{O}_C(\sum_{i=1}^n m_ip_i) \cong \omega_C^k \right\}, \]
where we drop the superscript when $k = 1$. We will reduce our problem to the case $n=1$ and study pointed Brill-Noether loci, defined as
\[ W^r_{d,a}(C,p) \coloneqq \left\{ L\in \textrm{Pic}^d(C) \ | \ h^0(C,L(-a_i\cdot p)) \geq r+1-i \ \forall \ 0\leq i \leq r \right\},\] 
where $a$ is a sequence $0\leq a_0 < a_1 < \cdots < a_r \leq d$. 

This study is motivated by the many applications of Brill-Noether Theory to understanding the birational geometry of moduli spaces of curves. When $C$ is generic in $\mathcal{M}_g$, understanding these loci has led to several important results about the birational geometry of this moduli space. Brill-Noether loci were instrumental in proving that $\mathcal{M}_g$ is unirational when $g \leq 14$ (see \cite{SeveriKod}, \cite{Sernesi81}, \cite{ChanRan84} and \cite{VerraUnirational}), uniruled and rationally connected for $g= 15$ (see \cite{SchreyerM15} and \cite{BrunoVerra}), and of general type for $g \geq 22$ (see \cite{KodMg}, \cite{KodevenHarris1984}, \cite{EisenbudHarrisg>23} and \cite{FarPayneJensen}). Moreover, Brill-Noether loci encode a lot of information about the geometry of the curve: both Green's Conjecture and the Gonality Conjecture can be deduced from Brill-Noether theoretic conditions, see \cite{aproduGreen-gonality}.

Moreover, studying pointed Brill-Noether loci for the underlying curve of a generic element in $\mathcal{H}_g^k(\mu)$ has already led to several results about the birational geometry of the projectivized strata of differentials. When $k = 1$, $g\leq 11$ and $\mu$ has only positive entries, many of the strata are uniruled or unirational, see \cite{BAR18} and \cite{Budstrata}. Similarly, when $k = 2, g \leq 6$ and $\mu$ has only positive entries, many of the strata are uniruled, see \cite{BAR18} and \cite{Budstrata}. We also know that many strata are of general type for $k= 1, g\geq 12$, see \cite{DaweiKodstrata}. 



\paragraph{\textbf{Intersection of Brill-Noether divisors with projectivized strata of differentials}}
In this paper, we show that projectivized strata of differentials are not contained in many of the possible Brill-Noether divisors, and that most Brill-Noether loci are of expected dimension. Because we have a clutching morphism 
\[ \mathcal{H}^k_g\text{\large(}k\cdot (2g-2)\text{\large)} \rightarrow \overline{\mathcal{H}}_g^k(\mu), \]
it is sufficient to study the pointed Brill-Noether loci of a generic element in $\mathcal{H}_g^k\text{\large(}k\cdot (2g-2)\text{\large)}$. A specialization argument will imply that $\mathcal{H}_g^k(\mu)$ is not contained in many of the pointed Brill-Noether divisors in $\mathcal{M}_{g,n}$. 

We will treat in this paper the cases $k=1$ and $k=2$ of Abelian and quadratic differentials, respectively. When $k=1$, $g\geq 4$ and $\mu$ is positive with all entries even, the stratum $\mathcal{H}_g(\mu)$ has two nonhyperelliptic components depending on the parity of the spin structure. There are some slight differences in the behaviour of strata of differentials with respect to Brill-Noether loci depending on the spin parity. These differences are noted in Theorem \ref{mainodd} and Theorem \ref{maineven}. The problem of understanding the irreducible components of strata of differentials received considerable attention in recent years, in particular for the cases $k=1$ and $k = 2$, see \cite{KontsevichZorich}, \cite{Boissy}, \cite{Lanneau}, \cite{ChenM} and \cite{chengendron}.

For the component   
\[\mathcal{H}_g^{\textrm{odd}}(2g-2) \coloneqq \left\{ [C,p] \in \mathcal{M}_{g,1} \ | \ \mathcal{O}_C\text{\large(}(2g-2)p\text{\large )} \cong \omega_C \ \textrm{and} \ h^0\text{\Large(}C,\mathcal{O}_C\text{\large(}(g-1)p\text{\large )}\text{\Large)} \ \textrm{is odd} \right\}, \]
we have the following result:

\begin{trm} \label{mainodd}
	Let $r\geq 1$, $g\geq r+2$ be natural numbers and $ a = (0\leq a_0<a_1<\cdots a_r \leq g+r)$ a vanishing sequence satisfying 
	\[ \rho(g,r,g+r, a)\coloneqq g - \sum_{i=0}^{r}(a_i-i) \leq -1\]
	and $a_r\leq g$. Then, if $[C,p]$ is generic in the stratum $\mathcal{H}_g^{\textrm{odd}}(2g-2)$, it does not admit a $g^r_{g+r}$ with prescribed vanishing $a$ at $p$. 
\end{trm}

We have a similar theorem regarding pointed Brill-Noether conditions and the even component of the minimal stratum. 

\begin{trm} \label{maineven}
	Let $r\geq 1$, $g\geq r+3$ be natural numbers and $ a = (0\leq a_0<a_1<\cdots a_r \leq g+r)$ a vanishing sequence satisfying 
	\[ \rho(g,r,g+r, a)\coloneqq g - \sum_{i=0}^{r}(a_i-i) \leq -1\]
	and $a_r\leq g-2$. Then, if $[C,p]$ is generic in the stratum $\mathcal{H}_g^{\textrm{even}}(2g-2)$, it does not admit a $g^r_{g+r}$ with prescribed vanishing $a$ at $p$. 
\end{trm}

We remark that the restriction on the vanishing order $a_r$ is very weak: If $a_0 = 0$, then the curve $[C,p]$ would admit a $g^{r-1}_{g+r-1}$ with vanishing sequence $a_1-1 < a_2-1 <\cdots < a_r-1$ at $p$. By repeating this procedure, we can assume $a_0\geq 1$ without loss of generality. The conditions $a_0 \geq 1$ and $\rho(g,r,g+r,a) = -1$ immediately imply $a_r \leq g+1$. In fact, Theorem \ref{mainodd} tells us that the only Brill-Noether divisor in $\mathcal{M}_{g,1}$ containing $\mathcal{H}_g^{\textrm{odd}}(2g-2)$ is the Weierstrass divisor, see Lemma \ref{only-Weierstrass}. 

We have similar results for the case when the pointed Brill-Noether number is positive, see Theorem \ref{posmainodd} and Theorem \ref{posmaineven}. As seen in Corollary \ref{ab-expected-dim}, these imply that all Brill-Noether loci
\[ W^r_{d}(C) \coloneqq \left\{ L\in \textrm{Pic}^d(C) \ | \ h^0(C,L) \geq r+1 \right\}\] 
 are of expected dimension when $C$ is the underlying generic curve appearing in $\mathcal{H}_g(\mu)$. This is in stark contrast to the case of Brill-Noether loci for generic curves in Hurwitz loci, where the gonality is relevant in determining both the dimension and the number of components, see \cite{Larson2Voigt} and the references therein. 

We will prove our results by degenerating our pointed curves to the boundary of $\mm_{g,1}$. For such an argument to work, we need to understand what pointed curves appear in the degeneration of strata of differentials, see \cite{FP18}, \cite{DaweiAbComp}, \cite{Daweik-diffcomp} and \cite{Daweimulti-scale}, as well as to understand the degenerations on the Brill-Noether loci, see \cite{limitlinearbasic}. 

\paragraph{\textbf{Maximal Brill-Noether loci and the Auel-Haburcak Conjecture}}

The minimal strata $\mathcal{H}^{\textrm{odd}}_g(2g-2)$ and  $\mathcal{H}^{\textrm{even}}_g(2g-2)$ are the loci parametrizing Weierstrass points having vanishing orders $(0,1,2,\ldots, g-2, 2g-2)$ and $(0,1,2,\ldots, g-3, g-1, 2g-2)$ respectively (cf. \cite{subcanbul}). Our results imply that these loci are not contained in many of the pointed Brill-Noether divisors in $\mathcal{M}_{g,1}$, and consequently not contained in many higher codimension pointed Brill-Noether loci. The interaction between various distinct Brill-Noether conditions is a topic that received considerable attention in recent years, with a lot of focus on the unmarked case, see  \cite{Haburcak-Auel}, \cite{Teixidor-BN} and \cite{Haburcak-Larson}. Theorem \ref{mainodd} and Theorem \ref{maineven} can be interpreted as a pointed version of this study. The strength of this result comes from the fact that one Brill-Noether locus is of high codimension $g$, while the other is only of codimension $1$ in $\mathcal{M}_{g,1}$. In fact, Theorem \ref{mainodd} can be used to study the Auel-Haburcak Conjecture. We obtain two non-containment criterions for maximal Brill-Noether loci in $\mathcal{M}_g$, see Theorem \ref{maxBNcrit1} and Theorem \ref{maxBNcrit2}. In particular, this implies the Auel-Haburcak Conjecture for $g  = 21$.

\paragraph{ \textbf{Birational geometry}}

Theorem \ref{mainquad} can be used to study the birational geometry of projectivized quadratic strata:
\begin{trm} \label{birational}
	Let $\mu = (m_1,\ldots, m_n)$ be a positive partition of $20$. Then $\mathcal{Q}_6(\mu)$ is uniruled. 
\end{trm} 
Here, we denoted 
	  \[ \mathcal{Q}(\mu) = 
\begin{cases} 
	\mathcal{H}^2_g(\mu) \setminus \mathcal{H}_g(\frac{\mu}{2})& \textrm{if all entries of} \ \mu \ \textrm{are even} \\
	\mathcal{H}^2_g(\mu) & \mathrm{otherwise.} 
\end{cases}
\]

The proof relies on a method used in \cite{BAR18}, \cite{Budstrata} and based on an idea that first appeared in \cite{FarVer}. 
If $[C,x_1,\ldots,x_n]$ is a generic point of $\mathcal{Q}_6(\mu)$ and $C$ is embedded as a hyperplane section for the linear system $|-2K_S|$ of a del Pezzo surface $S \in \mathcal{P}_4$ in $\mathbb{P}^{15}$, then there exists a $13$ dimensional linear subspace $\Lambda_{\mu}$ of $\mathbb{P}^{15}$ satisfying 
\[ \Lambda_{\mu}\cdot S = \sum_{i=1}^n m_ix_i \]
By intersecting the pencil of hyperplanes containing $\Lambda_{\mu}$ with the del Pezzo surface $S$ we get a rational curve 
\[\mathbb{P}^1 \rightarrow \overline{\mathcal{Q}}_6(\mu)\]
passing through $[C,x_1,\ldots, x_n]$, hence implying uniruledness. 

In genus $6$ the restriction $l(\mu)\geq 4$ on the length of the partition in \cite{Budstrata} appears as it was not known whether a generic point of $[C,x_1,\ldots, x_n]$ in the stratum $\mathcal{Q}_6(\mu)$ admits such an embedding in a del Pezzo surface when $l(\mu)\leq 3$. We will use a criterion of Mukai to conclude that such $[C]$ can be embedded in a del Pezzo surface in $\mathcal{P}_4$, see \cite{Mukgrass}. 

Putting together \cite[Theorem 4.3]{Barthesis}, \cite[Proposition 4.1]{Budstrata} and Theorem \ref{birational} we obtain that for every genus $3\leq g \leq 6$ and every positive partition $\mu$ of $4g-4$, all components of $\mathcal{Q}_g(\mu)$ are uniruled. We expect this to be a complete description of all strata of quadratic differentials not of general type. The Kodaira dimension of strata seems to depend on the genus but not on the positive partition $\mu$; and for genus $g =7$ we have that $\mathcal{Q}_7(1^{24})$ is of general type as there is a finite map from it to $\mathcal{M}_{7,17}$, which is of general type, see \cite{Logan} and \cite{Farkosz}.   

\paragraph{ \textbf{Prym curves and strata of quadratic differentials}}

One important motivation in studying pointed Brill-Noether condition is that they naturally appear as degenerations of the unpointed case. It was seen in \cite{BudPBN} and \cite{BudPrymIrr} that in fact pointed Brill-Noether conditions appear naturally as degenerations of Prym-Brill-Noether conditions. 

In Section \ref{s:prym}, we study the case when all entries of $\mu$ are even, and hence we have a map $\mathcal{Q}_g(\mu) \rightarrow \mathcal{R}_g$. We prove that for $[C,\eta]$ generic in the image, all Prym-Brill-Noether loci are of expected dimension. Similarly, we prove that the image is not contained in several divisors in $\cR_g$.  

\textbf{Acknowledgements:} I am very grateful to Ignacio Barros, Gavril Farkas, Richard Haburcak and Martin M\"oller for their helpful comments on this paper. The author acknowledges support by Deutsche Forschungsgemeinschaft
(DFG, German Research Foundation) through the Collaborative Research
Centre TRR 326 \textit{Geometry and Arithmetic of Uniformized Structures}, project number 444845124.
\section{Strata of differentials and Brill-Noether divisors}

In this section, we will study the position of projectivized strata of $k$-differentials with respect to different Brill-Noether loci. We will focus on the cases $k=1$ and $k=2$. Our first observation is that we can focus on the unique cases $\mu = (2g-2)$ and $\mu = (4g-4)$ respectively: For any positive partition $\mu = (m_1,m_2,\ldots, m_n)$ of $2g-2$ we have a clutching map 
\[ \mathcal{H}_g(2g-2)\rightarrow \mathcal{H}_g(\mu) \]
sending a curve $[C,p]$ to $[C\cup_p \mathbb{P}^1, p_1,\ldots, p_n]$, where the points $p_i$ are on the rational component. Because the marked curve $[\mathbb{P}^1, p, p_1,\ldots, p_n]$ is pointed Brill-Noether general, many non-inclusions can be decided at the level of the stratum $\mathcal{H}_g(2g-2)$. 

For a generic element $[C,p]$ in $\mathcal{H}_g(2g-2)$, we are interested to say when the pointed Brill-Noether locus 
  \[ W^r_{d,a}(C,p) \coloneqq \left\{ L\in \textrm{Pic}^d(C) \ | \ h^0(C,L(-a_i\cdot p)) \geq r+1-i \ \forall \ 0\leq i \leq r \right\}.\] 
is empty. Its expected dimension is the pointed Brill-Noether number \[\rho(g,r,d,a) \coloneqq g-(r+1)(g-d+r) - \sum_{i=0}^r(a_i-i).\] 
Our goal is to prove that a generic $[C,p] \in \mathcal{H}_g(2g-2)$, which is a very special curve in $\mathcal{M}_{g,1}$, behaves in many respects as a generic one, i.e. for almost all numerical choices $\rho(g,r,d,a) < 0$, the locus $W^r_{d,a}(C,p)$ is empty. 

For this, we study $g^r_d$'s with vanishing sequence $a = (0\leq a_0<a_1<\cdots a_r \leq d)$ at $p$. These are pairs $(V,L)$ where $L\in \textrm{Pic}^d(C)$ and $V\subseteq H^0(C,L)$ is $(r+1)$-dimensional, satisfying the conditions 
\[ \textrm{dim}\text{\Large (}V\cap H^0\text{\large(}C,L(-a_i\cdot p)\text{\large)}\text{\Large)} \geq r+1-i.\]
The motivation for studying such pairs $(V,L)$ instead of just line bundles $L$, is that they behave well with respect to degeneration to curves of compact type.

We denote by $\left\{C_i\right\}_{i \in I}$ the irreducible components of a curve $X$ of compact type. The degeneration of a $g^r_d$ to $X$ is a family $\left\{(V_i, L_i) \ | \ (V_i, L_i) \ \textrm{is a} \ g^r_d \ \textrm{on} \ C_i\right\}_{i \in I}$ satisfying certain ramification conditions at the nodes of $X$: \newline
$(\star)$ If $y$ is a node of $X$ connecting $C_i$ and $C_j$, and the vanishing sequences of $(V_i, L_i)$ and $(V_j, L_j)$ at $y$ are $(b_0,\ldots, b_r)$ and $(c_0,\ldots, c_r)$, then we have the inequalities 
\[ b_k + c_{r-k} \geq d \ \textrm{for every} \ 0\leq k \leq r. \]

To conclude Theorem \ref{mainodd} and Theorem \ref{maineven}, we will repeatedly use degenerations to such curves and induct on the genus.  

The stratum $\mathcal{H}_g(2g-2)$ has two non-hyperelliptic components, determined by the parity of the spin structure. As seen in Theorem \ref{mainodd} and Theorem \ref{maineven}, they behave slightly different with respect to Brill-Noether conditions. 

\textbf{Proof of Theorem \ref{mainodd}:} It is clear that we can assume 
	\[ \rho(g,r,g+r, a)\coloneqq g - \sum_{i=0}^{r}(a_i-i) = -1.\]
We prove the theorem in three steps: 
\begin{enumerate}
	\item We first show the result for $r = 1$ and $a_1 = g$.
	\item We assume the theorem holds for  $a_r = g$ in all genera and show it holds for $a_r<g$ too. 
	\item We use induction on $r$. We assume the theorem holds for $r-1$ and prove it holds for $r$. 
\end{enumerate}

\textbf{Step 1:} When $r = 1$ and $a_1 = g$, the condition 
	\[ \rho(g,r,g+r, a)\coloneqq g - \sum_{i=0}^{r}(a_i-i) = -1\]
implies $a_0= 2$. Assume $[C,p]$ admits a $g^1_{g+1}$ with prescribed vanishing $(2,g)$ at $p$ and let $L$ the associated line bundle of such a limit linear series. Then there exists $x\in C$ such that $L = \OO_C(g\cdot p +x)$. 

Because $a_0 = 2$, we have $h^0\text{\Large(}C,\OO_C\text{\large(}(g-2)\cdot p + x \text{\large)}\text{\Large)}\geq 2$. By Riemann-Roch this implies \[h^0\text{\large(}C,\OO_C(g\cdot p - x)\text{\large)}\geq 2.\]
Because $h^0\text{\large(}C,\OO_C(g\cdot p)\text{\large)}=  2$, the point $p$ is a base-point of $|g\cdot p|$. This implies $x = p$ and we get the contradiction 
\[1 = h^0\text{\Large(}C,\OO_C\text{\large(}(g-1)\cdot p  \text{\large)}\text{\Large)}\geq 2.\] 

\textbf{Step 2:} Given $r$, we assume the theorem holds in any genus whenever we have the equality $a_r = g$. We want to prove the theorem also holds when $a_r < g$. 

We prove this by induction on the genus. When $g = r+2$, there is no vanishing sequence $a$ satisfying  
\[ \rho(g,r,g+r, a) = -1 \ \textrm{and} \ a_r < g\]
Since $a_r = r+2 = g$, this initial case is completely covered by our assumption. 

We assume the theorem is true in genus $g-1$ and prove it in genus $g$. Let $[E,p,q]$ be a marked elliptic curve satisfying that $p-q$ is exactly $(2g-2)$-torsion. We consider the morphism 
\[\mathcal{H}_{g-1}^{\textrm{odd}}(2g-4) \rightarrow \overline{\mathcal{H}}_g^{\textrm{odd}}(2g-2)\]
sending a curve $[X,y]$ to $[X\cup_{y\sim q}E,p]$. 
This map is well-defined: Using \cite[Theorem 1.3]{DaweiAbComp}, we see that $[X\cup_{y\sim q}E,p]\in \overline{\mathcal{H}}_g^{\textrm{odd}}(2g-2)$. 

Assume by contradiction that the theorem does not hold for genus $g$. Then a generic $[C,p] \in \mathcal{H}_g^{\textrm{odd}}(2g-2)$ admits a $g^r_{g+r}$ with prescribed vanishing $a$ at $p$, where the vanishing $a$ satisfies the conditions in the hypothesis. Even more, $a_r < g$ by our assumption. 

As a consequence, the curve $[X\cup_{y\sim q}E, p]$ admits a limit $g^r_{g+r}$ with prescribed vanishing $a$ at $p$. Let $l_X$ and $l_E$ be the aspects of this limit linear series. 

Let $0\leq b_0 < b_1< \cdots < b_r \leq g+r$ be the vanishing orders of $l_E$ at $q$, and $0\leq c_0 < c_1< \cdots < c_r \leq g+r$ be the vanishing orders of $l_X$ at $y$. 

Because $a_r< g< 2g-2$ and $p-q$ is exactly $(2g-2)$-torsion, it follows that $a_i+b_{r-i} = g+r$ for at most one value $0\leq i \leq r$. Then the vanishing orders respect the conditions 
\[ a_i + b_{r-i} \leq g+r \ \textrm{and} \ a_j+b_{r-j} \leq g+r-1 \ \textrm{for} \ j\neq i \]
for some value $0\leq i\leq r$. 

The glueing conditions $c_j + b_{r-j} \geq g+r$ imply 
\[ c_i\geq a_i \ \textrm{and} \ c_j\geq a_j+1 \ \textrm{for} \ j\neq i \] 

To conclude Step 2, it is sufficient to show a generic $[X,y]\in \mathcal{H}_{g-1}^{\textrm{odd}}(2g-4)$ does not admit a $g^r_{g+r}$ with prescribed vanishing at $y$ given by 
\[ a^i \coloneqq (a_0+1-\delta^i_0,\ldots, a_j+1-\delta^i_j,\ldots, a_r+1-\delta^i_r) \]
for any $0\leq i \leq r$. Here we denoted $\delta^i_j$ to be the Kronecker delta. 

If $a_0+1-\delta^i_0 > 0$, this is equivalent to showing that $[X,y]$ does not admit a $g^r_{g-1+r}$ with vanishing at $y$ given by
\[(a_0-\delta^i_0,\ldots, a_j-\delta^i_j,\ldots, a_r-\delta^i_r) \]
This follows from the induction hypothesis. 

If $a_0 + 1 -\delta^i_0 = 0$, this implies $a_0 = 0$ and $i=0$. We want to show that  $[X,y]$ does not admit a $g^r_{g+r}$ with prescribed vanishing at $y$ given by 
\[ (0,a_1+1\ldots, a_j+1,\ldots, a_r+1). \]
Assume that it does and let $L_X$ be its associated line bundle. Then we have $h^0(X, L_X(-y)) \geq r+1$ and hence $[X,y]$ admits a $g^r_{g-1+r}$ with prescribed vanishing at $y$ given by $(0,a_1\ldots, a_j,\ldots, a_r)$. This is not possible because of the induction hypothesis. 

We conclude that $[C,p]$ does not admit a $g^r_{g+r}$ with prescribed vanishing $a$ at $p$. This contradicts our assumption that it does. This concludes Step 2. 

\textbf{Step 3:} We assume the proposition to be true for $r-1$ and we prove it for $r$. Because of \textbf{Step 2}, we can assume that $a_r = g$. Assume that $[C,p]$ admits a $g^r_{g+r}$ with prescribed vanishing $a$ at $p$ where $a$ satisfies 
\[ \rho(g,r,g+r, a) = -1 \ \textrm{and} \ a_r = g.\]
Denote by $L$ the line bundle associated to such a $g^r_{g+r}$. We distinguish two cases, depending on the value of $a_0$. 

If $a_0 = 0$, we look at $L(-p)$ and see that $[C,p]$ admits a $g^{r-1}_{g+r-1}$ with prescribed vanishing $(a_1-1,\ldots, a_r-1)$ at $p$. This is not possible because of the induction hypothesis. 

If $a_0 \geq 1$, because $a_r = g$ the only possibility for the vanishing sequence $a$ is 
\[ a = (1,2,\ldots, r-1, r+1, g).\]
There exists a degree $r$ effective divisor on $C$ such that 
\[ L = \OO_C(g\cdot p + D_r). \]
We have that
\[h^0\text{\Large(}C,\OO_C\text{\large(}(g-1)\cdot p + D_r  \text{\large)}\text{\Large)}\geq r+1\] 
By Riemann-Roch, this implies 
\[h^0\text{\Large(}C,\OO_C\text{\large(}(g-1)\cdot p - D_r  \text{\large)}\text{\Large)}\geq 1\]
The divisor $D_r$ is in the base locus of $|(g-1)\cdot p|$ and hence $D_r = r\cdot p$. We reach the contradiction  
\[1 = h^0\text{\Large(}C,\OO_C\text{\large(}(g-1)\cdot p  \text{\large)}\text{\Large)} = h^0\text{\Large(}C,L\text{\large(}-(r+1)\cdot p  \text{\large)}\text{\Large)}\geq 2\] 
This concludes Step 3. It is clear that the three steps imply the theorem.
\hfill $\square$

In fact, this theorem implies that the only Brill-Noether divisor containing $\mathcal{H}_g^{\textrm{odd}}(2g-2)$ is the Weierstrass divisor. This is an immediate consequence of the following lemma. 

\begin{lm} \label{only-Weierstrass} Let $g\geq 3$ and $1\leq r \leq g-1$. We consider the locus of pointed curves $[C,p] \in \cM_{g,1}$ admitting a $g^r_{g+r}$ with vanishing orders at $p$ equal to $(1,2,\ldots, r, g+1)$. This locus is the Weierstrass divisor. 
\end{lm}

\begin{proof}
	Let $L$ be the underlying line bundle of such a $g^r_{g+r}$. Then, there exists an effective divisor $D$ of degree $r-1$ such that we have 
	\[ L \cong \OO_C\text{\large(} (g+1)p+D\text{\large)}\]
	From the hypothesis, we have $h^0\text{\large(}C, L(-p)\text{\large)} \geq r+1$.
   
   Using the Riemann-Roch Theorem we deduce 
   \[ h^0\text{\large(}C, \omega_C(-gp-D)\text{\large)}\geq 1\]
   In particular, $p$ is a Weierstrass point of $C$, hence the conclusion. 
\end{proof}

Theorem \ref{maineven} can be immediately be obtained from Theorem \ref{mainodd}.

\textbf{Proof of Theorem \ref{maineven}:}

	Let $[E,p,q]$ be a marked elliptic curve satisfying that $p-q$ is exactly $(g-1)$-torsion. We consider the morphism 
	\[\mathcal{H}_{g-1}^{\textrm{odd}}(2g-4) \rightarrow \overline{\mathcal{H}}_g^{\textrm{even}}(2g-2)\]
	sending a curve $[X,y]$ to $[X\cup_{y\sim q}E,p]$. Because $a_r < g-1$ and $p-q$ is exactly $(g-1)$-torsion, we can apply Step 2 in the proof of Theorem \ref{mainodd}. 
	
	Hence, Theorem \ref{maineven} is simply a consequence of Theorem \ref{mainodd}. 
\hfill $\square$

For the projectivized strata of quadratic differentials, we have a more general result. 

\begin{trm} \label{mainquad}
	Let $r\geq 1$, $g\geq 3$ natural numbers and $ a = (0\leq a_0<a_1<\cdots a_r \leq g+r)$ a vanishing sequence satisfying 
\[ \rho(g,r,g+r, a)\coloneqq g - \sum_{i=0}^{r}(a_i-i) \leq -1.\]
 Then, a generic element $[C,p]$ in the stratum $\mathcal{Q}_g(4g-4)$ does not admit a $g^r_{g+r}$ with prescribed vanishing $a$ at $p$. 
\end{trm} 
\begin{proof} It is enough to prove the theorem for vanishing sequences satisfying 
	\[ \rho(g,r,g+r, a) = -1.\]
	We prove the theorem in three steps: 
	\begin{enumerate}
		\item We show that if the theorem holds for $r\leq g$, it also holds for $r\geq g+1$. 
		\item We prove the theorem for $g = 3$. 
		\item We prove the result in general by induction on the genus. 
	\end{enumerate}
    
    \textbf{Step 1:} Assume the theorem holds for $r-1\geq g$ and we prove it for $r+1$. 
    
    If $r\geq g+1$ and $a_0 \geq 1$, the condition 
	\[ \rho(g,r,g+r, a) = -1\]
	cannot be satisfied. Hence, we can assume $a_0 = 0$. If $[C,p]$ admits a $g^r_{g+r}$ with prescribed vanishing $a$ at $p$, then it admits a $g^{r-1}_{g+r-1}$ with prescribed vanishing at $p$ given by $(a_1-1,\ldots, a_r-1)$. Hence, the conclusion follows from the inductive step. 
	
	\textbf{Step 2:} We assume that a generic element of $[C,p]$ of $\mathcal{Q}_3(8)$ admits a $g^r_{3+r}$ with ramification profile $a$ at $p$ satisfying 
		\[ \rho(3,r,3+r, a) = -1.\]
	We denote by $L$ the associated line bundle to this linear series.
	
	We distinguish different cases depending on the value of $a_r$.
	
	If $a_r = r+1$, the unique choice of the vanishing sequence $a$ is 
	\[ a = (0,\ldots, r-4, r-2, r-1, r, r+1)\]
	This is not possible. We have the contradiction
	\[ 3 = h^0\textrm{\Large(}C, L\textrm{\large(}-(r-2)\cdot p\textrm{\large)}\textrm{\Large)} \geq 4 \]
	If $a_r = r+2$, we have two possible choices of $a$. The vanishing sequence is either $(0,\ldots, r-3, r-1, r, r+2)$ or $(0,\ldots, r-2, r+1, r+2)$. In particular, there exists $x\in C$ such that $L = \OO_C\textrm{\large(}(r+2)\cdot p+x\textrm{\large)}$. 
    
    When $a = (0,\ldots, r-3, r-1, r, r+2)$, we have 
    \[ h^0\textrm{\Large(}C, L\textrm{\large(}-(r-1)\cdot p\textrm{\large)}\textrm{\Large)} = h^0\textrm{\large(}C, \OO_C(3\cdot p + x)\textrm{\large)} \geq 3.\]
    Hence $3\cdot p + x \equiv K_C$. Using that $8\cdot p \equiv 2K_C$, we get $2\cdot p\equiv 2\cdot x$. This is impossible since $C$ is not hyperelliptic, see \cite[Proposition 3.3]{Budstrata}. 
    
    When $a = (0,\ldots, r-2, r+1, r+2)$, the condition 
    \[ h^0\textrm{\Large(}C, L\textrm{\large(}-(r+1)\cdot p\textrm{\large)}\textrm{\Large)} \geq 2 \]
    implies that $C$ is hyperelliptic, which we know is false. 
    
    If $a_r = r+3$, we have a unique choice for the vanishing sequence. We have 
    \[ a = (0,\ldots, r-2, r, r+3).\]
    This implies 
    \[ h^0\textrm{\large(}C, L(-r\cdot p)\textrm{\large)} = h^0\textrm{\large(}C, \OO_C(3\cdot p)\textrm{\large)} \geq 2.\] 
    By Riemann-Roch we have 
    \[ h^0\textrm{\large(}C, \omega_C(-3p)\textrm{\large)} \geq 1.\]
    Hence there exists $x\in C$ such that $3p+x\equiv K_C$. We obtain the same contradiction $2\cdot p\equiv 2\cdot x$ as before.
    
    \textbf{Step 3:} We assume the theorem holds for genus $g-1 \geq 3$ and prove it for genus $g$. We consider $[E,p,q]$ a marked elliptic curve satisfying that $p-q$ is exactly $(4g-4)$-torsion. Then we have a morphism 
    \[ \mathcal{Q}_{g-1}(4g-8) \rightarrow \mathcal{Q}_g(4g-4) \]
    sending a curve $[X,y]$ to $[X\cup_{y\sim q}E,p]$. This morphism is well-defined, see \cite[Theorem 1.5]{Daweik-diffcomp}. Because $a_r \leq g+r \leq 2g-1 < 4g-4$ and $p-q$ is exactly $(4g-4)$-torsion, the reasoning used in the proof of Theorem \ref{mainodd} can be used to reach the conclusion.  
\end{proof} 

The same is true for projectivized strata of meromorphic differentials with at least two negative entries.  
\begin{trm}
	Let $g\geq 0$, $r\geq 1$, $m_1, m_2 \geq 1$ natural numbers and $ a^j = (0\leq a^j_0<a^j_1<\cdots a^j_r \leq g+r)$ for $1\leq j \leq 3$ vanishing sequences satisfying 
	\[ \rho(g,r,g+r, a^1, a^2, a^3)\coloneqq g - \sum_{j=1}^{3}\sum_{i=0}^{r}(a^j_i-i) \leq -1.\]
	Then, if $[C,p_1,p_2,p_3]$ is generic in the stratum $\mathcal{H}^{\textrm{odd}}_g(2g-2+m_1+m_2, -m_1, -m_2)$, it does not admit a $g^r_{g+r}$ with prescribed vanishings $a^j$ at $p_j$ for $1\leq j \leq 3$. 
\end{trm} 
\begin{proof}
	Let $[\mathbb{P}^1, p_1,p_2,p_3, q_1,\ldots, q_g] \in \mathcal{H}_0(2g-2+m_1+m_2, -m_1,-m_2, -2,\ldots, -2)$ and let $\varphi$ the associated meromorphic differential (unique up to multiplication with a constant). We choose 
	\[[\mathbb{P}^1, p_1,p_2,p_3, q_1,\ldots, q_g] \in \mathcal{H}_0(2g-2+m_1+m_2, -m_1,-m_2, -2,\ldots, -2)\] 
	so that we have $\textrm{Res}_{q_i}(\varphi) = 0$ for all $1\leq i \leq g$. Th\'eor\`eme 1.2 of \cite{Taharabresidues} guarantees the existence of such a choice. 
	
	We glue an elliptic curve $E_i$ to each of the points $q_i$. We denote the curve obtained in this way by $[X,p_1,p_2,p_3]$ We have  $[X,p_1,p_2,p_3] \in \overline{\mathcal{H}}^{\textrm{odd}}_g(2g-2+m_1+m_2, -m_1, -m_2)$, see \cite[Theorem 1.3]{DaweiAbComp}. 
	
	Finally, we assume the theorem does not hold for some choice of $r$ and $a^1, a^2, a^3$. In particular, this implies that $[X,p_1,p_2,p_3]$ admits a limit $g^r_{g+r}$ with prescribed vanishing $a^j$ at $p_j$ for $1\leq j \leq 3$. We denote by $L_0$ the $\mathbb{P}^1$ aspect of this limit linear series, and similarly we denote $L_i$ for the $E_i$-aspect. 
	
	Brill-Noether additivity implies 
	\[ -1 \geq \rho(g,r,g+r, a^1, a^2, a^3) \geq \rho(L_0, p_1,p_2,p_3, q_1,\ldots, q_g) + \sum_{k=1}^g \rho(L_k, q_k) \]
	It is a classical result that all terms on the right-hand side are greater or equal to 0, see \cite[Theorem 1.1]{EisenbudHarrisg>23}. This concludes our proof by contradiction. 
\end{proof}

\section{Positive pointed Brill-Noether number}

We are now interested in considering the case when the pointed Brill-Noether number $\rho(g,r,g+r,a)$ is non-negative. In this situation, this number is the expected dimension of $W^r_{g+r,a}(C,p)$. We show that for $[C,p]$ generic in either $\mathcal{H}^{\textrm{odd}}_g(2g-2)$ or $\mathcal{H}^{\textrm{even}}_g(2g-2)$, all components of $W^r_{g+r,a}(C,p)$ are of the expected dimension. The same method as in the proof of Theorem \ref{mainodd} can be used with slight modifications to prove these results. 
 
\begin{trm} \label{posmainodd}
	Let $r\geq 0$, $g\geq 3$ natural numbers and $ a = (0\leq a_0<a_1<\cdots a_r \leq g+r)$ a vanishing sequence satisfying 
	\[ \rho(g,r,g+r, a)\coloneqq g - \sum_{i=0}^{r}(a_i-i) \geq 0.\]
For $[C,p]$ a generic element in the stratum $\mathcal{H}_g^{\textrm{odd}}(2g-2)$, we have that all components of the pointed Brill-Noether locus $W^r_{g+r, a}(C,p)$ are of expected dimension $ \rho(g,r,g+r, a)$. 
\end{trm}
\begin{proof} 
    As in the proof of Theorem \ref{mainodd}, we consider the map 
	\[\mathcal{H}_{g-1}^{\textrm{odd}}(2g-4) \rightarrow \overline{\mathcal{H}}_g^{\textrm{odd}}(2g-2).\]
    
      If $[C] \in \mathcal{M}_3$ is a nonhyperelliptic curve, then we have 
    \[ W^1_3(C) = \left\{K_C-q\right\}_{q\in C} \ \textrm{and} \ W^2_4(C) = \left\{K_C\right\} .\]
    Based on these facts, we can immediately check the theorem for $g= 3$. 
    
    We want to use this case as the induction base. Because we have $a_r - a_0 < 2g-2$, the same inductive procedure as in the proof of Theorem \ref{mainodd} works in our situation. 

\end{proof}

Using again the morphism 
\[\mathcal{H}_{g-1}^{\textrm{odd}}(2g-4) \rightarrow \overline{\mathcal{H}}_g^{\textrm{even}}(2g-2)\]
appearing in Theorem \ref{maineven} we conclude that 
\begin{trm} \label{posmaineven}
	Let $r\geq 0$, $g\geq 4$ natural numbers and $ a = (0\leq a_0<a_1<\cdots a_r \leq g+r)$ a vanishing sequence satisfying 
	\[ \rho(g,r,g+r, a)\coloneqq g - \sum_{i=0}^{r}(a_i-i) \geq 0.\]
For $[C,p]$ a generic element in the stratum $\mathcal{H}_g^{\textrm{even}}(2g-2)$, we have that all components of the pointed Brill-Noether locus $W^r_{g+r,a}(C,p)$ are of expected dimension $\rho(g,r,g+r,a)$. 
\end{trm}
\begin{proof}
    Let $[E,p,q]$ be a marked elliptic curve satisfying that $p-q$ is exactly $(g-1)$-torsion. We consider the morphism 
    \[\mathcal{H}_{g-1}^{\textrm{odd}}(2g-4) \rightarrow \overline{\mathcal{H}}_g^{\textrm{even}}(2g-2)\]
    sending a curve $[X,y]$ to $[X\cup_{y\sim q}E,p]$. 
    
    If $a_r - a_0 < g-1$, then our result is simply a consequence of the previous theorem. 
    The only situation when $a_0 > 0$ and $a_r \geq g-1 + a_0$ is when $a_i = i + 1$ for $0\leq i\leq r-1$ and $a_r = g$. But in this case, we can check that 
    \[  W^r_{g+r,(1,2,\ldots, r, g)}(C,p) = \left\{ \OO_C((g+r)\cdot p)\right\}\]
    This has expected dimension $0$, thus concluding the proof.
\end{proof}

As a consequence, many pointed Brill-Noether loci will be of expected dimension. In particular, we have 
\begin{cor} \label{ab-expected-dim} Let $g\geq 3$ and $\mu = (m_1,m_2,\ldots, m_n)$ a positive partition of $2g-2$. Let $[C,p_1,p_2,\ldots,p_n]$ be a generic element of a nonhyperelliptic component of $\mathcal{H}_g(\mu)$. Then for every positive integers $r, d$ with $ d \leq g+r-1$ the Brill-Noether locus 
\[ W^r_d(C) \coloneqq \left\{L\in \textrm{Pic}^d(C) \ | \ h^0(C,L)\geq r+1\right\} \]
is of expected dimension $\rho(g,r,d) \coloneqq g-(r+1)(g-d+r)$.
\end{cor}

As before, we have a similar result for quadratic strata. In this situation, the problem reduces again to considering the genus $3$ case. We obtain 

\begin{trm} \label{posmainquad}
	Let $r\geq 0$, $g\geq 3$ natural numbers and $ a = (0\leq a_0<a_1<\cdots a_r \leq g+r)$ a vanishing sequence satisfying 
\[ \rho(g,r,g+r, a)\coloneqq g - \sum_{i=0}^{r}(a_i-i) \geq 0.\]
For $[C,p]$ a generic element in the stratum $\mathcal{Q}_g(4g-4)$, we have that all components of the pointed Brill-Noether locus $W^r_{g+r, a}(C,p)$ are of expected dimension $ \rho(g,r,g+r, a)$. 
\end{trm} 

As a consequence, many pointed Brill-Noether loci will be of expected dimension. In particular, we have 

\begin{cor} \label{quad-expected-dim} Let $g\geq 3$ and $\mu = (m_1,m_2,\ldots, m_n)$ a positive partition of $4g-4$. Let $[C,p_1,p_2,\ldots,p_n]$ be a generic element of a nonhyperelliptic component of $\mathcal{Q}_g(\mu)$. Then for every positive integers $r, d$ with $ d \leq g+r-1$ the Brill-Noether locus 
	\[ W^r_d(C) \coloneqq \left\{L\in \textrm{Pic}^d(C) \ | \ h^0(C,L)\geq r+1\right\} \]
	is of expected dimension $\rho(g,r,d) \coloneqq g-(r+1)(g-d+r)$.
\end{cor}

\section{Maximal Brill-Noether loci}

We are looking at the Brill-Noether loci 
\[ \cM^r_{g,d} \coloneqq \left\{ [C]\in \cM_g \ | \ C \ \textrm{admits a} \ g^r_d\right\} \]
satisfying $-r-1 \leq \rho(g,r,d) \leq -1$. The Auel-Haburcak Conjecture says that, except for some low genus cases for $g = 7,8,9$, these Brill-Noether loci are maximal with respect to containments of Brill-Noether loci. 

In this section, we study the Auel-Haburcak Conjecture using strata of differentials. For this, we want to understand what are the higher codimension Brill-Noether loci containing $\mathcal{H}_g^{\textrm{odd}}(2g-2)$. 

\begin{prop} \label{highercodim}
	Let $r\geq 1$, $g\geq r+2$ natural numbers and $ a = (0\leq a_0<a_1<\cdots a_r \leq g+r)$ a vanishing sequence satisfying 
	\[ \rho(g,r,g+r, a)\coloneqq g - \sum_{i=0}^{r}(a_i-i) \leq -k\]
	and $a_r\leq g + k-1$. Then, if $[C,p]$ is generic in the stratum $\mathcal{H}_g^{\textrm{odd}}(2g-2)$ it does not admit a $g^r_{g+r}$ with prescribed vanishing $a$ at $p$. 
\end{prop}
\begin{proof}
	Without any loss of generality, we can assume $ \rho(g,r,g+r, a) = -k$. We will prove the theorem by induction on the rank $r$. 
	
	For $r = 1$, we have $a_0 + a_1 = g +k+1$. Moreover, if $a_1 \leq g$, then Theorem \ref{mainodd} implies the conclusion. We can assume $ g +1 \leq a_1 \leq g + k-1$. This implies $ 2 \leq a_0 \leq k$. In particular, we have $(a_0,a_1) \geq (2, g)$. Hence, if a generic $[C,p] \in \mathcal{H}^{\textrm{odd}}_g(2g-2)$ admits $g^1_{g+1}$ with prescribed vanishing $a$ at $p$, then it admits a $g^1_{g+1}$ with prescribed vanishing $(2,g)$ at $p$, contradicting Theorem \ref{mainodd}. 
	
	Next, we assume the theorem to hold for rank $r-1$ and prove it for rank $r$. If $a_0 = 0$ and $[C,p]$ admits a $g^r_{g+r}$ with prescribed vanishing $a$ at $p$, then $[C,p]$ admits a $g^{r-1}_{g+r-1}$ with prescribed vanishing $(a_1-1, a_2-1,\ldots, a_r-1)$ at $p$ and we can apply the induction hypothesis. 
	
	We are left to treat the case $a_0\geq 1$. In particular, this implies $a_i \geq i+1$ for every $0\leq i\leq r$. The conditions $\rho(g,r,g+r,a) = -k$ and $a_r\leq g+k-1$ imply $a_{r-1} \geq r+1$. Moreover, we can assume $a_r \geq g+1$. We have the inequality 
	\[ (a_0, a_1, \ldots, a_r) \geq (1,2,\ldots, r-1, r+1, g).\]
	In particular, if $[C,p]$ admits a $g^r_{g+r}$ with prescribed vanishing $(a_0,\ldots, a_r)$ at $p$, then it admits a $g^r_{g+r}$ with prescribed vanishing $(1,2,\ldots, r-1,r+1, g)$ at $p$. This contradicts Theorem \ref{mainodd}. 
\end{proof}

When $ k = r+1$, the condition $ a_r \leq g+r$ is automatically satisfied. We obtain that
\begin{cor} \label{technical-cor}
	Let $r\geq 1$, $g\geq r+2$ natural numbers and $ a = (0\leq a_0<a_1<\cdots <a_r \leq g+r)$ a vanishing sequence satisfying 
\[ \rho(g,r,g+r, a)\coloneqq g - \sum_{i=0}^{r}(a_i-i) \leq -r-1.\]
 Then, a generic $[C,p]$ in the stratum $\mathcal{H}_g^{\textrm{odd}}(2g-2)$ does not admit a $g^r_{g+r}$ with prescribed vanishing $a$ at $p$. 
\end{cor}

To obtain non-containment results for Brill-Noether loci, we will use a degeneracy argument. In order to control the vanishing orders at the node, we need the following lemma: 

\begin{lm} \label{techBN-lma}
	Let $[C,p] \in \mathcal{M}_{g,1}$ a pointed curve of genus $g$ and assume it admits a $g^r_d$ with ramification orders $(a_0,\ldots,a_r)$ at $p$, where $d > g +r$. Let $j$ be the maximal index for which $a_j < d-g-r + j$. Then the curve $[C,p]$ admits a $g^r_d$ with ramification orders 
	\[ (d-g-r, d-g-r+1,\ldots, d-g-r+j, a_{j+1}, a_{j+2}, \ldots, a_r)\]
	at the point $p$. 
\end{lm}
\begin{proof}
	Let $L$ be a line bundle of degree $d$ satisfying the conditions 
	\[ h^0(C, L(-a_i\cdot p)) \geq r+1-i. \]
	By Riemann-Roch, the conditions 
 	\[ h^0(C, L(-(d-g-r+i)\cdot p)) \geq r+1-i \]
 	are automatically satisfied. This immediately implies the conclusion.
\end{proof}

Next, we will prove non containments of Brill-Noether loci by using the fact that strata of differentials appear naturally in the boundary. Using Proposition \ref{highercodim} and Corollary \ref{technical-cor} we can show that such strata will appear in the boundary of one Brill-Noether locus, but not the other; hence implying the required non-containment. 

\begin{trm} \label{maxBNcrit1}
    Consider some positive integers $g,r,d,s,e$ satisfying the Brill-Noether numerical conditions 
    \[ -1 \geq \rho(g,r,d) \geq -r, \ \textrm{and} \ \rho(g,s,e) = -s-1 \]
    Then we have the non-inclusion $\mathcal{M}^r_{g,d}\not\subseteq \mathcal{M}^s_{g,e}$. 
\end{trm}

\begin{proof}
	Consider $r' = g+r-d-1$ and $d' = 2g-2-d$. Because of Serre duality we have the equality of loci $\mathcal{M}^r_{g,d} = \mathcal{M}^{r'}_{g,d'}$. We consider the clutching 
	\[ \mathcal{M}_{r'+1,1}\times \mathcal{M}_{d-r,1} \rightarrow \overline{\mathcal{M}}_g \] 
	glueing together the curve of genus $r'+1$ to the curve of genus $d-r$ at the marked points. 
	 
    We consider the locus of limit $g^{r'}_{d'}$ on the image of the clutching, given as \[\mathcal{G}^{r'}_{d',(d'-2r',d'-2r'+1,\ldots, d'-r'-1, d'-r'+|\rho(g,r,d)|)}\times \mathcal{G}^{r'}_{d',(r'-|\rho(g,r,d)|,r'+1,\ldots, 2r')}\]
    
     Our first goal is to show that 
    \[\mathcal{G}^{r'}_{d',(d'-2r',d'-2r'+1,\ldots, d'-r'-1, d'-r'+|\rho(g,r,d)|)}\times \mathcal{G}^{r'}_{d',(r'-|\rho(g,r,d)|,r'+1,\ldots, 2r')}\]
    is of expected dimension. Having this, we can use \cite[Corollary 3.5]{limitlinearbasic} to conclude the locus appear in the degeneration of $\mathcal{G}^{r'}_{d'}$ to the boundary of $\mm_g$.
    
    Notice that, once we take out the base locus $(d'-2r')p$, we obtain on $\mathcal{M}_{r',1}$ the locus of pointed curves admitting a $g^{r'}_{2r'}$ with ramification orders $(0,1,\ldots, r'-1, r'+ |\rho(g,r,d)|)$ at the point. The locus in $\mathcal{M}_{r'+1,1}$ admitting such a linear series is just the image of the stratum $\mathcal{H}_{r'+1}(r'+|\rho(g,r,d)|, \underbrace{1,\ldots,1}_{r'-|\rho(g,r,d)| \ \textrm{times}})$ when we forget all but the first point.

    Our next remark is that the locus $\mathcal{G}^{r'}_{d',(r'-|\rho(g,r,d)|,r'+1,\ldots, 2r')}$ dominates $\mathcal{M}_{d-r,1}$ if and only if we have $|\rho(g,r,d)| \leq r$: We look at the difference 
    \[d' - (d-r) - r' = (2g-2-d) -(d-r) - (g-d+r-1) = g-d-1 = r' -r.\] 
    If a curve admits a $g^{r'}_{d'}$ with ramification orders $(r'-r-1,r'+1,\ldots, 2r')$ at a point, then it admits a $g^{r'}_{d'}$ with ramification orders $(r'-r,r'+1,\ldots, 2r')$, see Lemma \ref{techBN-lma}. But the pointed Brill-Noether number of this condition is $-1$, and hence a generic curve in $\mathcal{M}_{d-r,1}$ does not admit such a linear series. 
    
    Using \cite[Corollary 3.5]{limitlinearbasic}, we get that the elements in \[\mathcal{G}^{r'}_{d',(d'-2r',d'-2r'+1,\ldots, d'-r'-1, d'-r'+|\rho(g,r,d)|)}\times \mathcal{G}^{r'}_{d',(r'-|\rho(g,r,d)|,r'+1,\ldots, 2r')}\]
    appear in the boundary compactification of $\mathcal{G}^{r'}_{d'}$ over $\overline{\mathcal{M}}_g$. 
    
    We consider a generic element $[C_1\cup_{p_1\sim p_2}C_2]$ in the image of 
    \[\mathcal{G}^{r'}_{d',(d'-2r',d'-2r'+1,\ldots, d'-r'-1, d'-r'+|\rho(g,r,d)|)}\times \mathcal{G}^{r'}_{d',(r'-|\rho(g,r,d)|,r'+1,\ldots, 2r')}\rightarrow \Delta_{r'+1}\subseteq \mm_g.\]
    If we assume $\mathcal{M}^r_{g,d}\subseteq \mathcal{M}^s_{g,e}$, it follows that this element $[C_1\cup_{p_1\sim p_2}C_2]$ admits a limit $g^s_e$, whose aspects we denote $l_1$ and $l_2$ and whose ramification orders at $p_1$ and $p_2$ we denote by $a = (a_0,\ldots, a_s)$ and $b= (b_0,\ldots, b_s)$. 
    
    Because $[C_2,p_2]$ is generic in $\mathcal{M}_{d-r,1}$, we must have $\rho(l_2,p_2) \geq 0$. The Brill-Noether additivity 
    \[ -s-1 \geq \rho(l_1,p_1) + \rho(l_2,p_2) \]
    implies $\rho(l_1,p_1) \leq -s-1$. But Corollary \ref{technical-cor} and Lemma \ref{techBN-lma} imply that $[C_1,p_1]$ does not admit a $g^s_e$ with ramification order $a$ at $p_1$ satisfying $\rho(r'+1,s,e,a) \leq -s-1$. This concludes our proof. 
\end{proof}

Using the same method, we can order the Brill-Noether loci in $\mathcal{M}_g$ by the value $2r + \rho(g,r,d)-d$. We can prove: 

\begin{trm} \label{maxBNcrit2}
    Consider some positive integers $g,r,d,s,e$ satisfying the Brill-Noether numerical conditions 
\[ -r \leq \rho(g,r,d) \leq -1, \ \rho(g,s,e) \leq -1 \ \textrm{and} \ d-2r + |\rho(g,r,d)| < e-2s +|\rho(g,s,e)| \]
Then we have the non-inclusion $\mathcal{M}^r_{g,d}\not\subseteq \mathcal{M}^s_{g,e}$. 	
\end{trm}
\begin{proof}
	Consider $r' = g+r-d-1$ and $d' = 2g-2-d$. Because of Serre duality we have the equality of loci $\mathcal{M}^r_{g,d} = \mathcal{M}^{r'}_{g,d'}$. We consider the clutching
	\[ \mathcal{M}_{g+2r'-d'-|\rho(g,r,d)|,1}\times \mathcal{M}_{d'-2r'+|\rho(g,r,d)|,1} \rightarrow \overline{\mathcal{M}}_g \] 
	and we denote $g_1 \coloneqq g+2r'-d' - |\rho(g,r,d)|$ and $g_2 \coloneqq d'-2r' +|\rho(g,r,d)|$. 
	
   We consider the locus of limit $g^{r'}_{d'}$ on the image of the clutching, given as \[\mathcal{G}^{r'}_{d',(d'-g_1-r'+1,\ldots, d'-g_1, d'-r'+|\rho(g,r,d)|)}\times \mathcal{G}^{r'}_{d',(r'-|\rho(g,r,d)|,g_1,\ldots, g_1+r'-1)}\]
   
   Notice that, once we take out the base locus $(d'-g_1-r'-1)p$, we obtain on $\mathcal{M}_{g_1,1}$ the locus of pointed curves admitting a $g^{r'}_{g_1+r'-1}$ with ramification orders $(0,1,\ldots, r'-1, g_1+ |\rho(g,r,d)|-1)$ at the point. 
   
   Let $[C,p, l]$ be an element in $\mathcal{G}^{r'}_{g_1+r'-1, (0,1,\ldots, r'-1, g_1+|\rho(g,r,d)|-1)}$ and let $L$ be the underlying line bundle of $l$. Then we have 
   \[ L \cong \OO_C\text{\large(}(g_1+|\rho(g,r,d)|-1)p + D \text{\large)} \]
   for some effective divisor $D$ of degree $r'-|\rho(g,r,d)|$. 
   By Serre duality, the condition $h^0(C, L) \geq r +1$ is equivalent to 
   \[ h^0(C, \omega_C(-D-(g_1+|\rho(g,r,d)|)p)) \geq 1 \]
   In particular, $[C,p]$ is in the image of the map 
   \[ \mathcal{H}_{g_1}(g_1+|\rho(g,r,d)|-1, 1, \ldots, 1) \rightarrow \mathcal{M}_{g_1,1} \]   
   forgetting all but the first marking. Moreover, 
   \[ L \cong \OO_C\text{\large(}(g_1+|\rho(g,r,d)|-1)p + D \text{\large)} \]
    where $K_C - D - (g_1+|\rho(g,r,d)|-1)p$ is equivalent to an effective divisor. 
    
   The locus $\mathcal{H}_{g_1}(2g_1-2)$ is contained in the image of the forgetful map
      \[ \mathcal{H}_{g_1}(g_1+|\rho(g,r,d)|-1, 1, \ldots, 1) \rightarrow \mathcal{M}_{g_1,1} \]   
      
     The discussion above implies the existence of a component of  $\mathcal{G}^{g_1-1}_{2g_1-1,(0,\ldots, g_1-2-r'. g_1-r',\ldots, g_1-1, 2g_1-2-r'+|\rho(g,r,d)|)}$, having dimension $3g-3-|\rho(g,r,d)|$ as expected, and whose image to $\mathcal{M}_{g_1,1}$ contains $\mathcal{H}^{\textrm{odd}}_{g_1}(2g_1-2)$. 
   
    We use \cite[Corollary 3.5]{limitlinearbasic} to conclude that the components of expected dimension of 
    \[\mathcal{G}^{r'}_{d',(d'-g_1-r'+1,\ldots, d'-g_1, d'-r'+|\rho(g,r,d)|)}\times \mathcal{G}^{r'}_{d',(r'-|\rho(g,r,d)|,g_1,\ldots, g_1+r'-1)}\]
   appear in the boundary compactification of $\mathcal{G}^{r'}_{d'}$ over $\overline{\mathcal{M}}_g$.  
   
   We consider a generic element $[C_1\cup_{p_1\sim p_2}C_2]$ in the image of 
   \[\mathcal{G}^{r'}_{d',(d'-g_1-r'+1,\ldots, d'-g_1, d'-r'+|\rho(g,r,d)|)}\times \mathcal{G}^{r'}_{d',(r'-|\rho(g,r,d)|,g_1,\ldots, g_1+r'-1)}\rightarrow \Delta_{g_1}\subseteq \mm_g\]
   If we assume $\mathcal{M}^r_{g,d}\subseteq \mathcal{M}^s_{g,e}$, it follows that this element $[C_1\cup_{p_1\sim p_2}C_2]$ admits a limit $g^s_e$, whose aspects we denote $l_1$ and $l_2$ and whose ramification orders at $p_1$ and $p_2$ we denote by $a = (a_0,\ldots, a_s)$ and $b= (b_0,\ldots, b_s)$. Without loss of generality, we can assume the limit linear series to be refined. Moreover, applying lemma \ref{techBN-lma} we can assume $b_0 \geq e-s-d+2r -|\rho(g,r,d)|$. This implies 
   \[ a_s \leq d+s + |\rho(g,r,d)|-2r\]
   After adding the base locus $(g_1+s-e) \cdot p_1$ we obtain that the curve $[C_1,p_1]$ admits a $g^s_{g_1+s}$ with ramification orders $(a_0+g_1+s-e,a_1+g_1+s-e,\ldots, a_s+g_1+s-e)$ at $p_1$. 
   
   If we show the inequality 
   \[ a_s + g_1+s-e \leq g_1 + |\rho(g,s,e)|-1 \]
   the conclusion follows from Proposition \ref{highercodim}. After subtracting $g_1$ from both sides, we are left to prove $a_s + s - e \leq |\rho(g,s,e)| -1$. But we have 
   \[ a_s + s-e \leq d + s +|\rho(g,r,d)|-2r + s-e \leq |\rho(g,s,e)| -1  \]
   because of the inequality $d-2r + |\rho(g,r,d)| < e-2s +|\rho(g,s,e)|$. This concludes the proof because of Proposition \ref{highercodim}. 
\end{proof}
\section{Uniruledness of quadratic strata in genus $6$}

In this section, we will prove uniruledness for all projectivized strata of quadratic differentials with positive partition in genus $6$. As mentioned in the introduction, this reduces to show that for a generic $[C,x_1,\ldots, x_n]$ in $\mathcal{Q}_6(\mu)$, the curve $C$ appears as a section in the linear system $|-2K_S|$ of a del Pezzo surface $S \in \mathcal{P}_4$ (i.e. $S$ is a blow-up of $\mathbb{P}^2$ in $4$ points).

\textbf{Proof of Theorem \ref{birational}:} We know from  \cite{Mukgrass} that for a curve $C$ to be contained in the linear series $|-2K_S|$ of a del Pezzo surface $S$, it suffices to know that 
	\begin{itemize}
		\item the curve $C$ has maximal gonality, 
		\item the curve $C$ is not bi-elliptic and 
		\item the curve $C$ is not a plane quintic. 
	\end{itemize}

Let $[C,p_1,\ldots,p_n] \in \mathcal{Q}_6(\mu)$ a generic element. We know from \cite[Proposition 3.2]{Budstrata} that $C$ has maximal gonality. We also know from Theorem \ref{mainquad} that $C$ is not a plane quintic. We are left to show that $C$ is not bi-elliptic. 

It is sufficient to show that for a generic element $[C,p] \in \mathcal{Q}_6(20)$, the curve $C$ is not bi-elliptic. 

Consider $[E_i, p_i,q_i] \in \mathcal{Q}_1(24-4i, 4i-24)$ for $1\leq i\leq 6$. We can assume that the six elliptic curves are all different. We consider the curve $X$ obtained by glueing $q_i$ to $p_{i+1}$ for $1\leq i \leq 5$. It follows from \cite[Theorem 1.5]{Daweik-diffcomp} that $ [X,p_1]\in \overline{\mathcal{Q}}_6(20)$.

We will use the theory of admissible covers, see \cite[Definition 4.1.1]{AbramovichCV}. Assume there exists an admissible double cover $\pi\colon X\rightarrow E$ from $X$ to a nodal curve $Y$ of genus $1$. 

We denote by $E\subseteq Y$ the unique elliptic component of $Y$. Because the $E_i$'s are all different, there is at most one of them satisfying $\pi(E_i) = E$. It follows that all the other elliptic components are mapped $2\colon 1$ to a rational component of $Y$. 

A $2\colon 1$ cover from an elliptic curve to a rational curve has exactly $4$ ramification points.

We denote by $E_i$ the component of $X$ mapped to $E$. For $j \neq i$, the map $\pi_{|E_j}\rightarrow \mathbb{P}^1$ has at least two ramification points that are not nodes of $X$ if $2\leq j \leq 5$, and at least three ramification points that are not nodes of $X$ if $j = 1$ or $j=6$. This gives 
\[ \sum_{j\neq i} \# (\textrm{ramifications of} \ \pi_{|E_j} \ \textrm{that are not nodes of} \ X) \geq 11 \] 

This is impossible. Using the Riemann-Hurwitz theorem and the definition of admissible covers, see \cite{AbramovichCV}, the sum must be equal to $10$. 

Hence, for a generic $[C,p_1,\ldots,p_n] \in \mathcal{Q}_6(\mu)$ the curve $C$ respects all three conditions. This implies uniruledness. 
 \hfill $\square$

\section{Quadratic strata and Prym-Brill-Noether loci} \label{s:prym}

The data $(C,\eta)$ of a smooth genus $g$ curve and a $2$-torsion line bundle on it is equivalent to the datum of an \'etale double cover $\pi\colon \widetilde{C}\rightarrow C$. To this double cover, we can associate its norm map 
\[ \mathrm{Nm}_\pi\colon \mathrm{Pic}^{2g-2}(\widetilde{C}) \rightarrow \mathrm{Pic}^{2g-2}(C)\]
and consequently we can define the Prym-Brill-Noether loci 
\[V^r(C,\eta) \coloneqq \left\{ L \in \mathrm{Pic}^{2g-2}(\widetilde{C}) \ | \ \mathrm{Nm}(L)\cong \omega_C, \ h^0(\widetilde{C}, L) \geq r+1, \ \mathrm{and} \ h^0(\widetilde{C}, L) \equiv r+1\ (\mathrm{mod} \ 2)  \right\}.\]
For a generic pair $[C,\eta] \in \mathcal{R}_g$, this Prym-Brill-Noether locus has dimension $g-1-\frac{r(r+1)}{2}$. Moreover, when $ g = \frac{r(r+1)}{2}$ the locus 
\[ \cR^r_g \coloneqq \left\{ [C,\eta] \in \cR_g \ | \ V^r(C,\eta) \neq \emptyset \right\} \]
is a divisor whose class is computed in \cite{BudPBN}. 

We are interested in the position of the quadratic strata with respect to this divisor, and moreover in the dimension of the Prym-Brill-Noether loci for a generic $[C,p_1,\ldots, p_n] \in \mathcal{Q}_g(\mu)$. To study this, we will use again a degeneration argument. We will use the compactification $\rr_g$ of $\cR_g$ appearing in \cite{Casa} and \cite{FarLud}. 

\begin{prop} \label{PBNtransv} Let $r\geq 3$, $g = \frac{r(r+1)}{2}$ and $\pi\colon \mathcal{Q}_g(4g-4) \rightarrow \mathcal{R}_g$ the map sending $[C,p]$ to the Prym curve \normalfont $[C, \omega_C\text{\large(}-(2g-2)\cdot p\text{\large)}]$. Then the image of $\pi$ is not contained in the Prym-Brill-Noether divisor $\mathcal{R}^r_g$.
\end{prop}

\begin{proof}
	We consider an element $[E,p,q] \in \mathcal{Q}_{1}(4g-4, 4-4g)$ whose associated quadratic differential $s$ satisfies $\textrm{Res}_q^2(s) = 0$. Such a choice of $[E,p,q]$ is guaranteed by \cite[Th\'eor\`eme 1.2]{GenTah-quad}. 
	We consider the map 
	\[ \mathcal{H}_{g-1}^{\textrm{odd}}(2g-4) \rightarrow \overline{\mathcal{Q}}_g(4g-4) \]
	sending a pointed curve $[X,y]$ to $[X\cup_{y\sim q}E,p]$. This map is well-defined, see \cite{Daweik-diffcomp}.
	
	The map $\pi\colon \mathcal{Q}_g(4g-4) \rightarrow \mathcal{R}_g$ extends to the boundary and sends the curve $[X\cup_{y\sim q}E,p]$ to the Prym curve $[X\cup_{y\sim q}E,\OO_X,\OO_E\text{\large(}(2g-2)\cdot (p-q)\text{\large)}]$. This curve is contained in the divisor $\Delta_1$ of $\rr_g$. 
	
	We know from \cite[Proposition 4.6]{BudPBN} that if $[X\cup_{y\sim q}E,\OO_X,\OO_E\text{\large(}(2g-2)\cdot (p-q)\text{\large)}]$ is contained in the Prym-Brill-Noether divisor, then $[X,y]$ is contained in the divisor  $\cM^{r-1}_{g-1,g+r-3}(a)$ parametrizing pointed curves $[C,p]$ admitting a $g^{r-1}_{g+r-3}$ with prescribed vanishing $a = (0,2,\ldots, 2r-2)$ at the point $p$. 
	
	Because $2r-2 \leq  \frac{r(r+1)}{2}-1$ for $r\geq 3$, Theorem \ref{mainodd} implies that a generic $[X,y]\in \mathcal{H}_{g-1}^{\textrm{odd}}(2g-4)$ does not admit such a $g^{r-1}_{g+r-3}$.  

	We conclude that $[X\cup_{y\sim q}E,\OO_X,\OO_E\text{\large(}(2g-2)\cdot (p-q)\text{\large)}]$ is not contained in $\rr^r_g$. 
\end{proof}

In fact, using the same method, we can prove an analogue of Corollary \ref{ab-expected-dim}:

\begin{prop} \label{PBN-expected} Let $g\geq 3$ and $\mu = (2m_1,2m_2,\ldots, 2m_n)$ a positive partition of $4g-4$. Let $[C,p_1,p_2,\ldots,p_n]$ be a generic element of $\mathcal{Q}_g(\mu)$. Then for every positive integer $r$ with $g-1-\frac{r(r+1)}{2} \geq 0$ the Prym-Brill-Noether locus $V^r(C, \omega_C(-\sum_{i=1}^nm_i\cdot p_i))$ is of expected dimension $g-1-\frac{r(r+1)}{2}$.
\end{prop}

\begin{proof}
	By a specialization argument, it is sufficient to prove this for $\mu = 4g-4$. We look at the map  
		\[ \mathcal{H}_{g-1}^{\textrm{odd}}(2g-4) \rightarrow \overline{\mathcal{Q}}_g(4g-4) \]
	as in the previous proposition and consider $[X\cup_{y\sim q}E,p]$ a generic element in its image. Looking at Prym-linear series and reasoning similarly to \cite{BudPBN} and \cite{BudPrymIrr} we obtain the inequality: 
    \[ \dim\text{\Large(}V^r\text{\large(}E\cup_{q\sim y}X,\OO_E\text{\large(}(2g-2)\cdot (p-q)\text{\large)}, \OO_X \text{\large)}\text{\Large)} \leq \dim\text{\large(}W^{r-1}_{g+r-3,a}(X,y)\text{\large)} \]
    where $ a = (0, 2 , \cdots , 2r-2)$.

But Theorem \ref{posmaineven} tells us that 
\[ \dim\text{\large(}W^{r-1}_{g+r-3,a}(X,y)\text{\large)} = \rho(g-1,r-1,g+r-3,a) = g-1-\frac{r(r+1)}{2}.\]
To conclude the proposition, for a generic $[C,p] \in \mathcal{Q}_g(2g-2)$ we have the double inequality 
\[ g-1-\frac{r(r+1)}{2} \leq V^r\text{\large(}C, \omega_C(-(2g-2)\cdot p)\text{\large)} \leq  g-1-\frac{r(r+1)}{2}\] 
and hence the conclusion. 
\end{proof}

Finally, we study the position of quadratic strata with respect to the Prym-Hurwitz divisor appearing in \cite[Theorem 0.2]{FarLud}. We consider the divisor 
\[ \mathcal{D}_{2i+1:2} \coloneqq \left\{[C,\eta]\in \mathcal{R}_{2i+1} \ | \ h^0(C,\wedge^iQ_{K_C}\otimes \eta) \geq 1 \right\}\]
    
It was remarked in \cite{FarLud} and proven in \cite{FarMust} that 
\[ \mathcal{D}_{2i+1:2} = \left\{[C,\eta] \in \mathcal{R}_{2i+1} \ | \ \eta \in C_i-C_i \ \subset \textrm{Pic}^0(C) \right\} \]
where $C_i-C_i$ denotes the $i$-th difference variety of $C$. In particular, this description realizes $\mathcal{D}_{2i+1:2}$ as a Hurwitz divisor and its intersection with the boundary divisor can be studied as in \cite{Bud-adm} and \cite{Bud-newdiv}. 

\begin{prop} \label{FarkasLudwig}
	 Let $g = 2i+1$ and $\pi\colon \mathcal{Q}_g(4g-4) \rightarrow \mathcal{R}_g$ the map sending $[C,p]$ to \normalfont $[C, \omega_C\text{\large(}-(2g-2)\cdot p\text{\large)}]$. Then the image of $\pi$ is not contained in the divisor $\mathcal{D}_{2i+1:2}$. 
\end{prop}
\begin{proof}
    We keep the setting of the proof of Proposition \ref{PBNtransv}. We show that  $[X\cup_{y\sim q}E,\OO_X,\OO_E\text{\large(}(2g-2)\cdot (p-q)\text{\large)}]$ is not contained in $\overline{\mathcal{D}}_{2i+1:2}$. 
    
    The fact that $E$ is not general plays no role in our study. The same method as in \cite[Proposition 4.4]{Bud-adm} and \cite[Proposition 5.4]{Bud-newdiv} can be used to prove that  $[X\cup_{y\sim q}E,\OO_X,\OO_E\text{\large(}(2g-2)\cdot (p-q)\text{\large)}]$ is contained in $\overline{\mathcal{D}}_{2i+1:2}$ if and only if $[X,y]$ admits a $g^1_{d}$ with prescribed vanishing $(0,2d-g+1)$ at $y$ for some $i+1\leq d \leq g-1$. We know from Theorem \ref{mainodd} that a generic $[X,y]$ does not admit such linear series.
\end{proof}

Let $g = 2i+1$ and $\mu = (m_1,\ldots, m_n)$ be a partition of $2g-2$. Proposition \ref{FarkasLudwig} implies that for a generic element $[C,p_1,\ldots,p_n]$ of $\mathcal{Q}_g(2\cdot\mu)$ we have 
\[ h^0\textrm{\Large (}C,\wedge^iQ_{K_C}\otimes \omega_C(-\sum_{i=1}^{n}m_i\cdot p_i)\textrm{\Large)}  = 0 \]
We consider the virtual divisor
\[ \mathcal{D}_g(\mu) \coloneqq \left\{[C,p_1,\ldots,p_n] \in \cM_{g,n} \ | \ h^0\textrm{\Large (}C,\wedge^iQ_{K_C}\otimes \omega_C(-\sum_{i=1}^{n}m_i\cdot p_i)\textrm{\Large)} \geq 1 \right\}. \]
This can be alternatively described as the image of the map 
\[ \mathcal{H}_g(\mu, \underbrace{1,\ldots,1}_{i \ \textrm{times}}, \underbrace{-1,\ldots,-1}_{i \ \textrm{times}}) \rightarrow \mathcal{M}_{g,n}  \]
forgetting the last $2i$ markings, and hence it is a divisor in $\mathcal{M}_{g,n}$ because of \cite[Theorem 1]{Bud}. 

We have the following immediate consequence of Proposition \ref{FarkasLudwig}. 

\begin{cor}
Let $\mu = (m_1,\ldots, m_n)$ be a partition of $2g-2$. Then the locus $\mathcal{Q}_g(2\cdot \mu)$ is not contained in the divisor $\mathcal{D}_g(\mu)$.
\end{cor}
\bibliography{main}
\bibliographystyle{alpha}
\Addresses
\end{document}